\documentclass[a4paper]{amsart} 

\usepackage{amsfonts}
\usepackage{amsmath}  
\usepackage{amscd}
\usepackage{graphicx}
\usepackage{psfrag}

\usepackage{amssymb}
\usepackage{latexsym}
\usepackage{framed, color}

\newtheorem{thm}{\sc Theorem}[section]

\newtheorem{lem}[thm]{\sc Lemma}
\newtheorem{prop}[thm]{\sc Proposition}
\theoremstyle{definition}

\theoremstyle{remark}
\newtheorem{exam}[thm]{\sc Example}
\newtheorem{rmk}[thm]{\sc Remark}
\newtheorem{prob}[thm]{\sc Problem}
\makeatletter
 
 \@addtoreset{equation}{section}
\makeatother

\newcommand{\R}{\mathbf{R}}

\newcommand{\Z}{\mathbf{Z}}

\newcommand{\Int}{\mathop{\mathrm{Int}}\nolimits}

\newcommand{\Image}{\mathop{\mathrm{Im}}\nolimits}

\renewcommand{\setminus}{\smallsetminus}
\def\spmapright#1{\smash{%
 \mathop{\hbox to 1cm{\rightarrowfill}}
  \limits^{#1}}}

\title{Linking between singular locus and regular fibers}
\dedicatory{Dedicated to Professor Goo
Ishikawa on the occasion of his 60th birthday}
\author{Osamu Saeki} 
\address{Institute of Mathematics for Industry,
Kyushu University,
Motooka 744, Nishi-ku, Fukuoka 819-0395, Japan}

\email{saeki@imi.kyushu-u.ac.jp}
\date{\today}

\begin{document}
\begin{abstract}
Given a null-cobordant oriented framed link $L$ in a closed 
oriented $3$--manifold $M$, we determine those links
in $M \setminus L$ which can be realized as the singular point
set of a generic map $M \to \R^2$ that has $L$
as an oriented framed regular fiber. Then, we study
the linking behavior between the singular point
set and regular fibers for generic maps of $M$ into $\R^2$.
\end{abstract}

\maketitle 

\section{Introduction}\label{section1}

Topology of generic $C^\infty$ maps of manifolds of dimension $\geq  2$
into the plane, $\R^2$, has been extensively studied as a natural generalization
of the Morse theory, which studies generic maps into the line, $\R$.
For a Morse function, singular points, or critical points, are isolated
and their positions in the source manifold are not interesting except for
their cardinalities or indices.
On the other hand, for a generic map into the plane, the singular point
set is a smooth submanifold of dimension one in the source
manifold and its position may be non-trivial.
In \cite{Saeki95}, the author studied the position of the
singular point set and characterized those smooth $1$--dimensional
submanifolds which arise as the singular point set of a generic map.

On the other hand, each regular fiber of such a generic map
into $\R^2$ is of codimension two and
is disjoint from the singular point set.
Therefore, the singular point set and regular fibers may be
non-trivially linked. 

In September 2018\footnote{In the conference 
``Geometric and Algebraic Singularity Theory'' held
in honor of the 60th birthday of Goo Ishikawa, 
in B\c{e}dlewo, Poland.} Professor David Chillingworth
asked the author the following question: 
\textit{for a generic map $f : \R^3 \to \R^2$,
must every component of a regular fiber
be linked by at least one component of 
the singular point set ?}

In this paper, we concentrate ourselves to generic maps of
\emph{closed} (i.e.\ compact and boundaryless) 
$3$--dimensional manifolds, instead of $\R^3$,
and study the linking behavior between the singular point set
and regular fibers in the source $3$--manifold.
More precisely, let $M$ be a closed oriented $3$--manifold and
$f : M \to \R^2$ a generic $C^\infty$ map. Generic maps
that we consider in this paper are called \emph{excellent maps},
as defined in \S\ref{section2}, and have fold and cusp singularities.
In our $3$--dimensional case, both the singular point set
and regular fibers have dimension one, and they constitute
disjoint links in $M$. We study their relative positions
in the $3$--manifold $M$.

For example, let us consider the unit sphere
$S^3 \subset \R^4$ and let
$\pi : \R^4 \to \R^2$ be the standard
projection defined by $\pi(x_1, x_2, x_3, x_4) = (x_1, x_2)$
for $(x_1, x_2, x_3, x_4) \in \R^4$.
Then, $f_0 = \pi|_{S^3} : S^3 \to \R^2$ is an excellent map
whose singular point set 
$S(f_0) = \{(x_1, x_2, 0, 0) \in S^3\}$
consists only of definite fold singularities and
is a trivial knot in $S^3$.
Furthermore, for $y = (y_1, y_2)$ with $y_1^2 + y_2^2 < 1$, the 
regular fiber $f_0^{-1}(y)
= \{(y_1, y_2, x_3, x_4) \in S^3\}$ is an unknotted
circle linked with $S(f_0)$ (see Fig.~\ref{fig1}).
So, in this example, the answer to the above
question is positive.

\begin{figure}[t]
\centering
\psfrag{f}{$f_0^{-1}(y)$}
\psfrag{s}{$S(f_0)$}
\includegraphics[width=0.9\linewidth,height=0.2\textheight,
keepaspectratio]{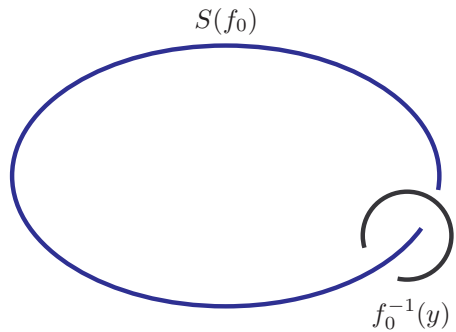}
\label{fig1}
\caption{Singular point set and a regular fiber for a
specific map $f_0 : S^3 \to \R^2$}
\end{figure}

The present paper is organized as follows.
In \S\ref{section2}, we will first see that regular fibers
are naturally oriented and framed; i.e.\ they have natural normal 
framings induced by the generic map $f : M \to \R^2$.
Furthermore, they bound compact oriented normally framed
surfaces embedded in $M$. Conversely, in \cite{Saeki94},
it has been shown that if an oriented normally framed
link in $M$ bounds a compact oriented normally framed surface, then
it is realized as a regular fiber of a generic map of $M$ into $\R^2$.
Then, in Theorem~\ref{thm:main}, given such a framed link $L$ in $M$,
we characterize those unoriented links in $M \setminus L$
that arise as the singular point set of a generic map $f : M \to \R^2$
such that $L$ coincides with a framed regular fiber of $f$.
The characterization is given in terms of a relative characteristic
class (see \cite{Ker}) which is the obstruction to
extending a certain trivialization of the tangent bundle of $M$
on a neighborhood of $L$ to the whole $M$.

In \S\ref{section3}, we will study the relative characteristic
class which arises as the obstruction as above.
As a consequence, we will show that if a regular
fiber has an odd number of components, then it
necessarily links with the singular point set (see
Remark~\ref{rmk:odd}).
We will also give a result which enables us to identify the obstruction
for local links that are embedded inside an open $3$--ball.

In \S\ref{section35},
by utilizing the results obtained in \S\ref{section3}, 
we show that there exist generic maps
$S^3 \to \R^2$ such that a regular fiber, which is a $2$--component
link, and the singular point
set are split; i.e.\ they lie inside disjoint $3$--balls.
We also see that there exists such an example for
every closed oriented $3$--manifold $M$.
We also give two explicit examples of generic maps
$S^3 \to \R^2$ which exhibit
non-linking phenomena between regular fibers and the singular
point set.

Finally in \S\ref{section4},
we address the original question concerning
generic maps of $\R^3$ into the plane.
By utilizing results obtained in \cite{HP} on regular fibers of
submersions $\R^3 \to \R^2$, we answer to
the question negatively, by constructing counter examples.

Throughout the paper, manifolds and maps
are differentiable of class $C^\infty$
unless otherwise indicated.
All (co)homology
groups are with $\Z_2$--coefficients unless otherwise
indicated. The symbol ``$\cong$'' means
an appropriate isomorphism between algebraic objects. 

\section{Main theorem}\label{section2}

Let $M$ be a closed oriented $3$--dimensional manifold.
We say that a map $f : M \to \R^2$ is \emph{excellent}
if its singularities consist only of
fold and cusp singularities, where
a \emph{fold singularity} (or a 
\emph{cusp singularity}) is modeled on
the map germ $(x, y, z) 
\mapsto (x, y^2 \pm z^2)$ (resp.\ $(x, y, z) 
\mapsto (x, y^3 + xy - z^2)$) at the origin.
We say that a fold singularity is \emph{definite} 
(resp.\ \emph{indefinite}) if it is modeled on the map germ
$(x, y, z) \mapsto (x, y^2 + z^2)$ (resp.\ $(x, y, z) 
\mapsto (x, y^2 - z^2)$).

It is known that the set of excellent maps
is always open and dense in the mapping space
$C^\infty(M, \R^2)$ endowed with the Whitney $C^\infty$
topology (for example, see \cite{GG, Whitney55}).

In the following, for a map $f : M \to \R^2$,
we denote by $S(f)$ the set of singular points of $f$.
If $f$ is an excellent map, then we see easily that
$S(f)$ is a link in $M$, i.e.\ a disjoint union
of finitely many smoothly embedded circles.
For a regular value $y \in \R^2$, if $L = f^{-1}(y)$
is non-empty, then we call it a \emph{regular fiber},
which is also a link in $M$ and is disjoint from $S(f)$.
We fix an orientation of $\R^2$ once and for all, and then 
a regular fiber is naturally oriented, since $M$ is oriented.
Furthermore, $L$ is naturally \emph{framed}: its framing
is given as the pull-back of the trivial normal framing
of the point $y$ in $\R^2$ (see Fig.~\ref{fig2}). 
In other words, taking
a small disk neighborhood of $y$ in $\R^2$ consisting
entirely of regular values, let $y'$
be a point in its boundary, then $f^{-1}(y')$ represents
the framed longitude of the framed link $L$.

\begin{figure}[t]
\centering
\psfrag{f}{$f$}
\psfrag{m}{$M$}
\psfrag{r}{$\R^2$}
\psfrag{y}{$y$}
\psfrag{yy}{$y'$}
\psfrag{fy}{$L = f^{-1}(y)$}
\psfrag{fyy}{$f^{-1}(y')$}
\includegraphics[width=\linewidth,height=0.35\textheight,
keepaspectratio]{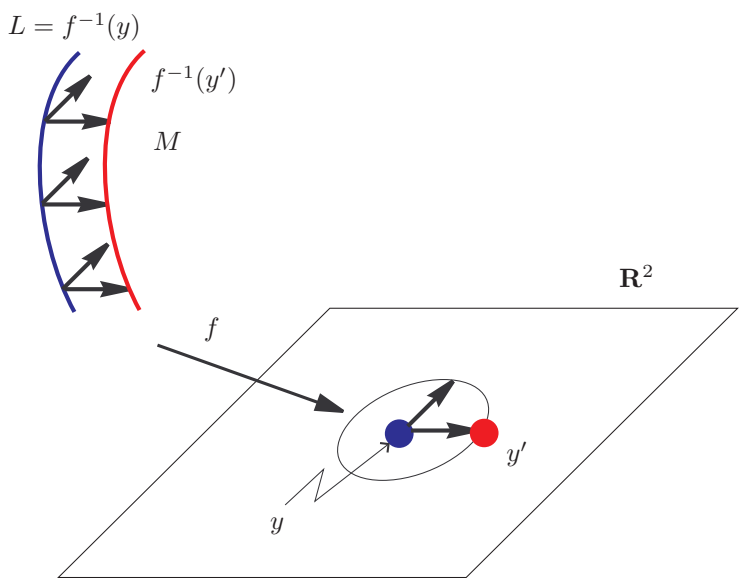}
\caption{Framing for a regular fiber}
\label{fig2}
\end{figure}

\begin{lem}
A framed regular fiber $L$ of an excellent map
$f : M \to \R^2$ over a regular point $y \in \R^2$
is always framed null-cobordant.
In other words, there exists a compact oriented
surface $V$ embedded in $M$ whose boundary coincides
with $L$ and which is consistent with the framed
longitude.
\end{lem}

\begin{proof}
Let $\ell$ be a half line in $\R^2$ emanating
from $y$. We may assume that it is transverse to
the map $f$. Then, $V = f^{-1}(\ell)$ gives
the desired surface (see Fig.~\ref{fig5}).
\end{proof}

\begin{figure}[t]
\centering
\psfrag{f}{$f$}
\psfrag{m}{$M$}
\psfrag{r}{$\R^2$}
\psfrag{y}{$y$}
\psfrag{fy}{$L = f^{-1}(y)$}
\psfrag{l}{$\ell$}
\psfrag{v}{$V = f^{-1}(\ell)$}
\includegraphics[width=\linewidth,height=0.4\textheight,
keepaspectratio]{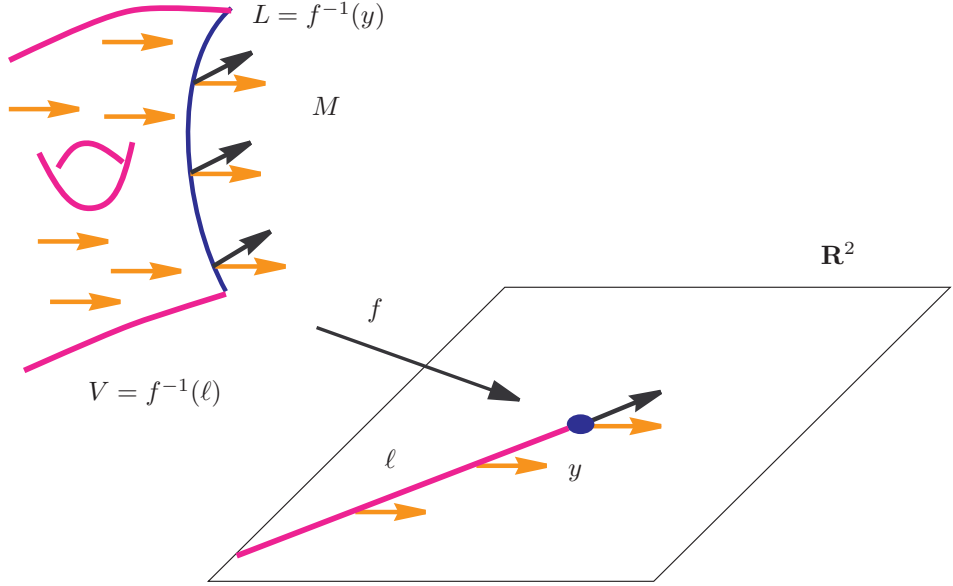}
\caption{Constructing a framed null-cobordism}
\label{fig5}
\end{figure}

In \cite[Proposition~5.1]{Saeki94}, it has been shown that
every null-cobordant oriented framed link $L$ in $M$ can
be realized as an oriented framed regular fiber
of an excellent map $f : M \to \R^2$.
In this case, the singular point set $S(f)$
is a link disjoint from $L$.
Then, it is natural to ask which links in $M \setminus L$
appear as the singular point set of such an excellent map.

In order to state our first theorem, let us
prepare some notations and terminologies.
For a (unoriented) link $J$ in $M \setminus L$, we denote
by $[J]_2 \in H_1(M \setminus L)$ the $\Z_2$--homology
class represented by $J$. Let $N(L)$ be a small tubular
neighborhood of $L$ in $M$ disjoint from $J$. 
Since $L$ is a framed link,
we have a natural $3$--framing of $M$ over $\partial N(L)$, i.e.\ a
trivialization of $TM|_{\partial N(L)}$.
The obstruction to extending this framing
over $M \setminus \Int{N(L)}$ is the relative
Stiefel--Whitney class (see \cite{Ker}), 
denoted by $w_2(M, L)$,
which is an element of the
$\Z_2$--cohomology group 
$$H^2(M \setminus \Int{N(L)},
\partial N(L)) \cong H^2(M, N(L)) \cong H^2(M, L),$$
where the first isomorphism is given by excision and
the second one is given by the natural homotopy
equivalence $(M, L) \to (M, N(L))$. Note that by
Poincar\'e--Lefschetz duality, we have 
$$H^2(M \setminus \Int{N(L)},
\partial N(L)) \cong H_1(M \setminus \Int{N(L)}) \cong
H_1(M \setminus L).$$

\begin{rmk}\label{rmk:w}
Let $j : (M, \emptyset) \to (M, L)$
be the inclusion. Then the induced homomorphism
$j^* : H^2(M, L) \to H^2(M)$ maps $w_2(M, L)$
to the second Stiefel--Whitney class $w_2(M)$
of $M$, which vanishes. By the cohomology exact sequence
$$H^1(L) \spmapright{\delta} H^2(M, L) \spmapright{j^*} H^2(M),$$
we see that $w_2(M, L) = \delta (\alpha)$ for some $\alpha
\in H^1(L)$.
\end{rmk}

Now, one of the main theorems of this paper is the following.

\begin{thm}\label{thm:main}
Let $L$ be an oriented null-cobordant framed link
in a closed oriented $3$--manifold $M$, and 
$J$ be an unoriented link in $M$ disjoint from $L$.
Then, there exist an excellent map $f : M \to \R^2$
and a regular value $y$ such that $f^{-1}(y)$
coincides with $L$ as an oriented framed link and that
$S(f) = J$ if and only if $[J]_2 \in H_1(M
\setminus L)$ is Poincar\'e dual to $w_2(M, L)
\in H^2(M, L)$.
\end{thm}

\begin{proof}
Suppose that $f : M \to \R^2$ is an excellent map
such that $L$ coincides with $f^{-1}(y)$ as a
framed link for a regular value $y$ and that
$J = S(f)$. Then, by Thom \cite{Thom55},
$[J]_2$ is Poincar\'e dual to $w_2(M, L)$.

Conversely, suppose that $[J]_2 \in H_1(M
\setminus L)$ is Poincar\'e dual
to $w_2(M, L)$. Let $g : M \to \R^2$
be an arbitrary excellent map for which
there exists a regular value $y$ such that
$g^{-1}(y)$ coincides with $L$ as a framed link.
Such an excellent map always exists by \cite{Saeki94}.
Then, we see that $[S(g)]_2 \in H_1(M
\setminus L)$ is Poincar\'e dual
to $w_2(M, L)$. By our assumption, this implies
that $J$ and $S(g)$ are $\Z_2$--homologous
in $M \setminus L$.

Then, by the same argument as in \cite{Saeki95}
applied to $M \setminus L$, which uses band operations
and cusp elimination techniques, we see that $g$ is
homotopic relative to $N(L)$ to an excellent map $f$
with $S(f) = J$.
This completes the proof.
\end{proof}

\begin{rmk}
As in \cite{Saeki95}, suppose $J$ is decomposed as
a disjoint union
$$J = J_0 \cup J_1 \cup C,$$
where $J_0$ and $J_1$ are finite disjoint unions
of open arcs and circles, $C$ is a finite set of points,
and each point of $C$ is adjacent to both $J_0$ and $J_1$.
If both $J_0$ and $J_1$ are non-empty, then
in Theorem~\ref{thm:main}, we can find an excellent map $f$
such that $S(f) = J$, $J_0$ is the set of definite fold
singularities, $J_1$ is the set of indefinite fold singularities,
and $C$ is the set of cusp singularities.
\end{rmk}

\begin{rmk}
Let $g : M \to \R^2$ be an excellent map
for which there exists a regular value $y$ such that
$g^{-1}(y)$ coincides with $L$ as a framed link.
In the situation of the theorem, we see that 
$[J]_2 \in H_1(M)$ is Poincar\'e dual to $w_2(M)$,
which vanishes,
by Remark~\ref{rmk:w}. Then, we can apply the 
modification techniques developed in \cite{Saeki95}
without touching $L$
to obtain an excellent map $h : M \to \R^2$
homotopic to $g$ such that $S(h)$ is isotopic to
$J$ in $M$. However, in order to obtain an excellent map
$h'$ such that $S(h')$ coincides with $J$,
we need to further modify $h$. In such 
a modification process, the regular
fiber over $y$ may change,
since in the course of the isotopy, the link may
cross $L$. In \S\ref{section3}, we will see that 
not every $\Z_2$ null-homologous link $J$ in $M$ can
be realized as above, depending on its position
relative to $L$.
\end{rmk}

Generalizing our Theorem~\ref{thm:main}, we can
also obtain the following, which can be proved
by the same argument. Details are left to the reader.

\begin{thm}
Let $M$ be a closed oriented $3$--manifold and
$L_1, L_2, \ldots, L_\ell$, and $J$
be disjoint links in $M$.
Suppose that $L_1, L_2, \ldots, L_\ell$
are oriented and null-cobordant framed links, and 
that they bound disjoint compact oriented framed surfaces.
Furthermore, $J$ is an unoriented link. 
Then, there exist an excellent map
$f : M \to \R^2$ and distinct regular values $y_1, y_2, \ldots, y_\ell
\in \R^2$ of $f$ such that $f^{-1}(y_i) = L_i$ as framed
links for $i = 1, 2, \ldots, \ell$,
and $J = S(f)$ if and only if
$[J]_2 \in H_1(M \setminus L)$
is Poincar\'e dual to $w_2(M, L)$, where
$L = L_1 \cup L_2 \cup \cdots \cup L_\ell$.
\end{thm}

For maps into $S^2$, we have a similar result as follows.
Recall that, for a closed oriented $3$--dimensional
manifold $M$, the homotopy classes of $M$ into $S^2$
are in one-to-one
correspondence with the framed cobordism classes
of closed oriented framed $1$--dimensional submanifolds
in $M$ by the Pontrjagin--Thom construction.
For the classification of the homotopy set $[M, S^2]$
for a closed oriented $3$--manifold $M$, the reader
is referred to \cite{CRS}.

\begin{thm}\label{thm3}
Let $M$ be a closed oriented $3$--manifold and
fix a homotopy class of a map $g : M \to S^2$.
Let $L$ be an oriented framed link in $M$ which
corresponds to the homotopy class of $g$.
Then, for an unoriented link $J$ in $M \setminus L$,
there exist an excellent map
$f : M \to S^2$ homotopic to $g$
and a regular value $y \in S^2$ of $f$ such that 
$f^{-1}(y)$ coincides with $L$ as a framed link
and $J = S(f)$ if and only if
$[J]_2 \in H_1(M \setminus L)$
is Poincar\'e dual to $w_2(M, L)$.
\end{thm}

The proof of Theorem~\ref{thm3} is similar to that
of Theorem~\ref{thm:main} and is left to the reader.
Note that Theorem~\ref{thm:main} corresponds to
the case of a null-homotopic map $g$ in Theorem~\ref{thm3}
in a certain sense.

\section{Obstruction}\label{section3}

In order to apply Theorem~\ref{thm:main} in 
practical situations,
let us study the obstruction class $w_2(M, L)$
more in detail, where $M$ is a closed oriented
$3$--manifold and $L$ is a framed link in $M$.

As we saw in Remark~\ref{rmk:w}, there exists an
$\alpha \in H^1(L)$ such that $\delta(\alpha) 
= w_2(M, L)$. 
In fact, the cohomology class $\alpha$ can
be explicitly given as follows.
It is known that a closed oriented $3$--manifold $M$
is always parallelizable, i.e.\ its
tangent bundle is trivial. Let us fix a framing $\tau$
of $M$, where $\tau$ can be identified with a
trivialization of the tangent bundle $TM$.
Once such a framing $\tau$ is fixed, we can compare it with
the specific framing given on each component $L_s$ of the framed link $L$. 
This defines a well-defined element $a_s$ in $\pi_1(SO(3)) \cong \Z_2$.
Then, we have the following.

\begin{lem}
The Kronecker product $\langle \alpha, [L_s]_2 \rangle \in \Z_2$
coincides with $a_s$.
\end{lem}

The above lemma follows from the definition
of the obstruction class $w_2(M, L)$. We omit the details.

For example, if the framing on $L$ coincides with $\tau$
up to homotopy, then $\alpha = 0$ and consequently
we have $w_2(M, L) = 0$.

Note that the framing $\tau$ may not be unique.
The set of homotopy classes of such framings
is in one-to-one correspondence with the
homotopy set $[M, SO(3)]$. If we consider
the set of homotopy classes of framings
on the $2$--skeleton of $M$, then each such
framing up to homotopy defines a \emph{spin
structure} on $M$, and the set of spin
structures is in one-to-one correspondence with
$H^1(M)$ (see \cite{Milnor}).

By the cohomology exact sequence,
$$H^1(M) \spmapright{i^*} H^1(L) 
\spmapright{\delta} H^2(M, L) \spmapright{j^*} H^2(M),$$
we see that for every element $\gamma \in \Image{i^*}$,
we could choose $\alpha + \gamma$ instead of $\alpha$,
where $i : L \to M$ is the inclusion map.
The observation in the previous paragraph shows that
this corresponds to choosing another framing, say $\tau'$,
which is ``twisted along $\gamma$''.

\begin{rmk}
As we saw in Remark~\ref{rmk:w}, $w_2(M, L)$
is in the kernel of $$j^* : H^2(M, L) \spmapright{} H^2(M),$$
which coincides with $\Image{\delta} \cong H^1(L)/\Image{i^*}$.
Note that this latter group is always non-trivial,
since $L$ bounds a compact surface in $M$ and
hence $[L]_2 = 0$ in $H_1(M)$. 

If we change the framing of a component $L_s$ of $L$, then
$w_2(M, L)$ changes in general. The difference
is described by $\delta[L_s]_2^*$, where $[L_s]_2^*$
is the dual to the homology class $[L_s]_2 \in H_1(L)$
represented by $L_s$ with respect to the basis
of $H_1(L)$ consisting of the homology classes
represented by the components of $L$.
This follows from the observation described in
\cite[pp.~520--521]{Ker}.
(However, we need to be careful, since if we change
the framing of $L_s$, then the resulting framed link may
not be framed null-cobordant any more.)
\end{rmk}

\begin{rmk}
Let $L$ be an oriented link in a closed oriented $3$--manifold $M$.
Then, we can easily show that it bounds a compact oriented
surface in $M$ if and only if $L$ represents zero
in $H_1(M; \Z)$.
\end{rmk}

In order to apply Theorem~\ref{thm:main} in practical
situations, we need to identify the obstruction $w_2(M, L)$.
First, we have the following.

\begin{prop}
Let $L$ be an oriented framed link which bounds a compact 
oriented surface $V$ consistent with the framing.
Let $\alpha \in H^1(L)$ be an element such that
$\delta(\alpha) = w_2(M, L)$.
Then, we have
\begin{eqnarray*}
\langle w_2(M, L), [V, \partial V]_2 \rangle & = &  
\langle \delta(\alpha), [V, \partial V]_2 \rangle \\
& = & \langle \alpha, [L]_2 \rangle \\
& \equiv & \chi(V) \pmod{2} \\
& \equiv & \sharp L \pmod{2},
\end{eqnarray*}
where $\langle \cdot \, , \, \cdot \rangle$ is the Kronecker product,
$[V, \partial V]_2 \in H_2(M, L)$ is the
fundamental class of $V$ in $\Z_2$--coefficients,
$\chi(V)$ denotes the Euler characteristic 
of $V$, and $\sharp L$ denotes the number of
components of $L$.
\end{prop}

The above proposition is similar to the
Poincar\'e--Hopf theorem for vector fields.
It can be proved by decomposing $V$ into
simplices, and by computing the contribution
of each simplex. We omit the details.

The above proposition can also be proved as follows.
First, we construct an excellent map $f : M \to \R^2$
such that for a regular value $y$, $f^{-1}(y)$
coincides with $L$ as a framed link and that for a half line $\ell$
emanating from $y$ in $\R^2$ transverse to $f$, 
we have $f^{-1}(\ell) = V$.
Such an excellent map is constructed in \cite{Saeki94}.
Then, the map $f|_V : V \to \ell$ is a Morse
function and its number of critical points
coincides with the number of intersection points
of $V$ and $S(f)$. Since $[S(f)]_2$ is Poincar\'e
dual to $w_2(M, L)$, we see that this number coincides with
$\langle w_2(M, L), [V, \partial V]_2 \rangle$. Since the
number of critical points of the Morse function
is congruent modulo $2$ to $\chi(V)$, we get the
result. The congruence $\chi(V) \equiv \sharp L \pmod{2}$
is obvious, since $V$ is a compact orientable surface
and $\partial V = L$.

\begin{rmk}\label{rmk:odd}
The above proposition shows the following.
If $f : M \to \R^2$ is an excellent map and
$y \in \R^2$ is a regular value such that
$L = f^{-1}(y)$ has an odd number of components,
then every compact oriented surface $V$
in $M$ bounded by $L$ compatible with
the framing of $L$ intersects with $S(f)$.
If $H_1(M) = 0$, then this implies that the
$\Z_2$ linking number of $L$ and $S(f)$ in $M$
does not vanish.
Thus, in this case, the regular fiber $L$ necessarily
links with $S(f)$ (see Fig.~\ref{fig7}).
\end{rmk}

\begin{figure}[t]
\centering
\psfrag{l}{$\ell$}
\psfrag{v}{$V = f^{-1}(\ell)$}
\psfrag{s}{$S(f)$}
\psfrag{F}{$L$}
\psfrag{r}{$\R^2$}
\psfrag{m}{$M$}
\psfrag{f}{$f$}
\psfrag{fs}{$f(S(f))$}
\psfrag{y}{$y$}
\includegraphics[width=\linewidth,height=0.45\textheight,
keepaspectratio]{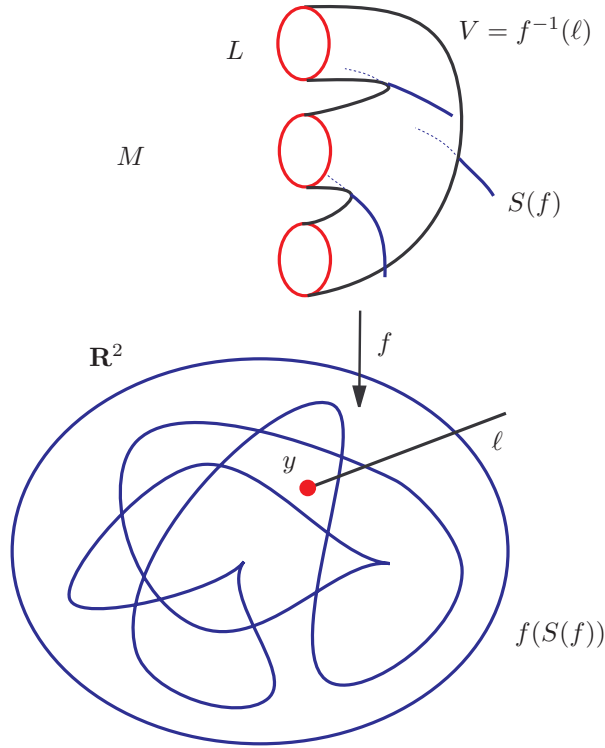}
\caption{Regular fiber with an odd number of components links with
the singular point set.}
\label{fig7}
\end{figure}

Let us now consider the case of a local knot
component. Suppose that the oriented framed link
$L$ contains a component $K$ that lies in
an open set $U$ in $M$ diffeomorphic to $\R^3$.
In the following, let us identify $U$ with $\R^3$.
In this case, up to homotopy, we may assume that
the framing $\tau$ for $M$ is given
by the standard framing of $\R^3$.

Let $\pi : \R^3 \to H$ be the orthogonal projection
onto a generic hyperplane $H \cong \R^2$
in the sense that $\pi|_K$ is an immersion with
normal crossings. On the other hand,
we may assume that the first vector field
defining the framing $\tau$ over $K$ is tangent to
$K$ consistent with the orientation.
Then, we can count the number of times modulo $2$ the
$2$--framing consisting of the remaining two vector
fields rotates vertically with respect
to $H$, which we denote by $t_v(K)$.
Then, we have the following.

\begin{lem}
Let $\alpha \in H^1(L)$ be an element such that
$\delta(\alpha) = w_2(M, L)$. Then,
we have
$$\langle \alpha, [K]_2 \rangle \equiv
t_v(K) + c(K) +1 \pmod{2},$$
where $c(K)$ denotes the number of
crossings of the immersion $\pi|_K : K \to H$
with normal crossings.
\end{lem}

\begin{proof}
Since the framing $\tau$ is standard on $U = \R^3$,
in order for the obstruction to vanish on $K$,
we need to have that the winding number of
$\pi(K)$ on $H$ is even as long as $t_v(K) = 0$. 
On the other hand,
by \cite{Whitney}, we have that the winding number
has the same parity as $c(K)$+1. 
Thus, by the observation in \cite[pp.~520--521]{Ker},
we have the conclusion.
\end{proof}

\section{Examples}\label{section35}

In this section, we give some explicit examples
which imply that the answer to the problem
posed in \S\ref{section1} for closed orientable $3$--manifolds
is negative in general.

\begin{exam}\label{exam:h1}
Let $L$ be a $2$--component framed link $h^{-1}(\{y_1, y_2\})$
in $S^3$ that consists of two framed fibers of the 
positive Hopf fibration
$h : S^3 \to S^2$, for $y_1 \neq y_2$ in $S^2$,
where we reverse
the orientation of one of the components
and the framings are induced by $h$.
By taking the inverse image $h^{-1}(a)$ of an embedded
arc $a$ in $S^2$ connecting $y_1$ and $y_2$,
we see that $L$ is framed null-cobordant (see Fig.~\ref{fig8}).
By the above lemma, we have
that $w_2(S^3, L)$ vanishes.
Therefore, by Theorem~\ref{thm:main},
an arbitrary link $J$ split from $L$ can be realized
as the singular point set of an excellent map
$S^3 \to \R^2$ with $L$ a framed regular fiber.
In this example, the components of the regular fiber $L$
do not link with the singular point set! 

\begin{figure}[t]
\centering
\psfrag{l}{$L$}
\psfrag{s3}{$S^3$}
\psfrag{s2}{$S^2$}
\psfrag{q1}{$y_1$}
\psfrag{q2}{$y_2$}
\psfrag{a}{$a$}
\psfrag{h}{$h$}
\includegraphics[width=\linewidth,height=0.3\textheight,
keepaspectratio]{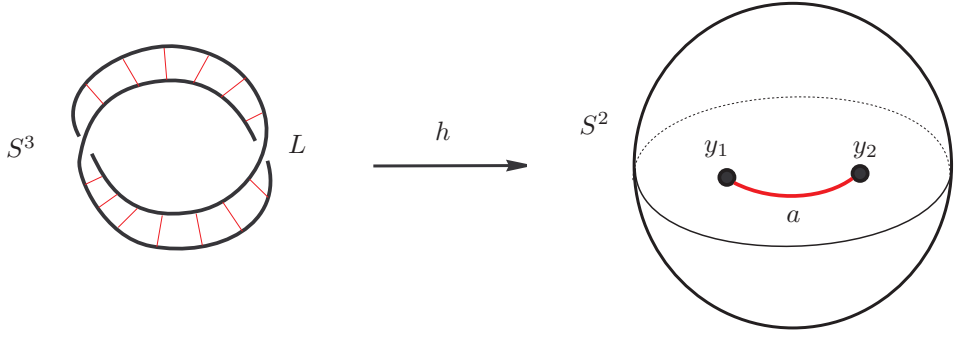}
\caption{Framed Hopf link which is null-cobordant}
\label{fig8}
\end{figure}

Note that $L$ has an even number of components.
This is consistent with the observation given in
Remark~\ref{rmk:odd}.

Let $M$ be an arbitrary closed oriented $3$--manifold.
By considering the above $2$--component
link $L$ as embedded in $\R^3 \subset S^3$ and by embedding it
to $M$, we get the same result for $M$ as well.
This gives
counter examples to the question presented in \S\ref{section1}
for closed orientable $3$--manifolds.
\end{exam}

We will give two explicit examples of excellent maps
on $S^3$ which give counter examples.

\begin{exam}\label{ex2}
Let $h : S^3 \to S^2$ be the (positive) Hopf
fibration. Let $p_N = (0, 0, 1)$ and $p_S = (0, 0, -1)$ be the north and
the south poles of $S^2$, respectively, where we identify
$S^2$ with the unit sphere in $\R^3$.
We decompose $S^2$ as
$S^2 = D_N \cup D_S \cup A$, where
$D_N$ (or $D_S$) is a small $2$--disk neighborhood
of $p_N$ (resp.\ $p_S$) in $S^2$ with $D_N \cap D_S =
\emptyset$, and $A$ is the annulus
obtained as the closure of $S^2 \setminus (D_N \cup D_S)$.

Note that the fibration $h$ is trivial on each of $D_N$,
$D_S$ and $A$. Let us fix a trivialization 
\begin{equation}
h^{-1}(A)
= S^1 \times A = S^1 \times ([-1, 1] \times S^1)
= (S^1 \times [-1, 1]) \times S^1,
\label{eqn1}
\end{equation}
where we identify $A$ with $[-1, 1] \times S^1$
so that $\{1\} \times S^1$ (or
$\{-1\} \times S^1$) coincides with
$\partial D_N$ (resp.\ $\partial D_S$).
We take the trivialization of $h^{-1}(A)$ in such a way that
it extends to a trivialization of $h$ over $D_N \cup A$.
Note that in (\ref{eqn1}), the first $S^1$--factor
corresponds to the fibers of $h$ and the last $S^1$--factor
corresponds to the equatorial direction of $S^2$ 
in the target.

Let $k : S^1 \times [-1, 1] \to [1, \infty)$ be a Morse function such that
\begin{itemize}
\item[(1)] $k^{-1}(1) = S^1 \times \{-1, 1\}$,
\item[(2)] $k$ has no critical point in a small neighborhood
of $S^1 \times \{-1, 1\}$,
\item[(3)] $k$ has exactly two critical points in such a way that
one of them has index $1$ and the other has index $2$.
\end{itemize}

Using the above ingredients,
let us now construct an excellent map
$f : S^3 \to \R^2$ as follows.
On $h^{-1}(D_N)$ (or on $h^{-1}(D_S)$), we define
$f = i_N \circ h$ (resp.\ $f = i_S \circ h$), 
where $i_N : D_N \to \R^2$ (resp.\ $i_S : D_S \to \R^2$)
is an orientation preserving (resp.\ reversing)
embedding onto the unit disk in $\R^2$
such that $i_N(p_N) = i_S(p_S)$ coincides with the
origin $\mathbf{0}$. Furthermore, we choose $i_N$ and $i_S$
such that for each $t \in S^1$, $i_N(1, t) = i_S(-1, t)$ holds
for $(1, t)$ and $(-1, t) \in [-1, 1] \times S^1 = A$.
On $h^{-1}(A) = (S^1 \times [-1, 1]) \times S^1$, 
we define $f$ by
$f(x, t) = \eta(k(x), t)$
for $x \in S^1 \times [-1, 1]$ and $t \in S^1$,
where $\eta : [1, \infty) \times S^1 \to \R^2$
is an embedding such that its image is the
complement of the open unit disk in $\R^2$
and that $\eta(\{1\} \times S^1)$ coincides with
the unit circle in $\R^2$. We choose $\eta$
consistently with $i_N$ and $i_S$, i.e.\ we require
the condition that $\eta(1, t) = i_N(1, t) = i_S(-1, t)$
for every $t \in S^1$. Then, the map $f : S^3 \to \R^2$ 
thus constructed is well-defined.

By modifying $f$ near the attached tori $h^{-1}(\partial D_N
\cup \partial D_S)$ appropriately,
we may assume that $f$ is a smooth excellent map.
Furthermore, the origin $\mathbf{0}$ of $\R^2$
is a regular value and $f^{-1}(\mathbf{0})$
is a framed regular fiber as in Example~\ref{exam:h1}.
Note that $S(f)$ has two components: one consists
of definite fold singularities and the other of indefinite fold 
singularities.

The situation is as depicted in Fig.~\ref{fig9}.
The torus in the top figure represents $h^{-1}(\{0\} \times S^1)$
for $\{0\} \times S^1 \subset [-1, 1] \times S^1 = A$,
and it separates the regular fiber components $h^{-1}(p_N)$
and $h^{-1}(p_S)$ of $f$. The annulus depicts $h^{-1}([1-\varepsilon, 1]
\times \{t\})$ for some small $\varepsilon > 0$ and for some $t \in S^1$.
We may assume that the critical points of $k$ on $h^{-1}([-1, 1] \times \{t\})$
are contained in $h^{-1}([1-\varepsilon, 1]
\times \{t\})$.
As $t$ varies in $S^1$ in the
positive direction, the annulus rotates as depicted in that figure.
Therefore, the critical points of $k$ on the annulus sweeps out
a $2$--component link $S(f) = S_0(f) \cup S_1(f)$ as
depicted in the bottom figure, where $S_0(f)$ (or $S_1(f)$)
is the set of definite (resp.\ indefinite) fold
singularities of $f$.

\begin{figure}[t]
\centering
\psfrag{h}{$h^{-1}(p_N)$}
\psfrag{hs}{$h^{-1}(p_S)$}
\psfrag{s0}{$S_0(f)$}
\psfrag{s1}{$S_1(f)$}
\psfrag{b}{$h^{-1}([1-\varepsilon, 1] \times \{t\})$}
\psfrag{T}{$h^{-1}(\{0\} \times S^1)$}
\includegraphics[width=\linewidth,height=0.5\textheight,
keepaspectratio]{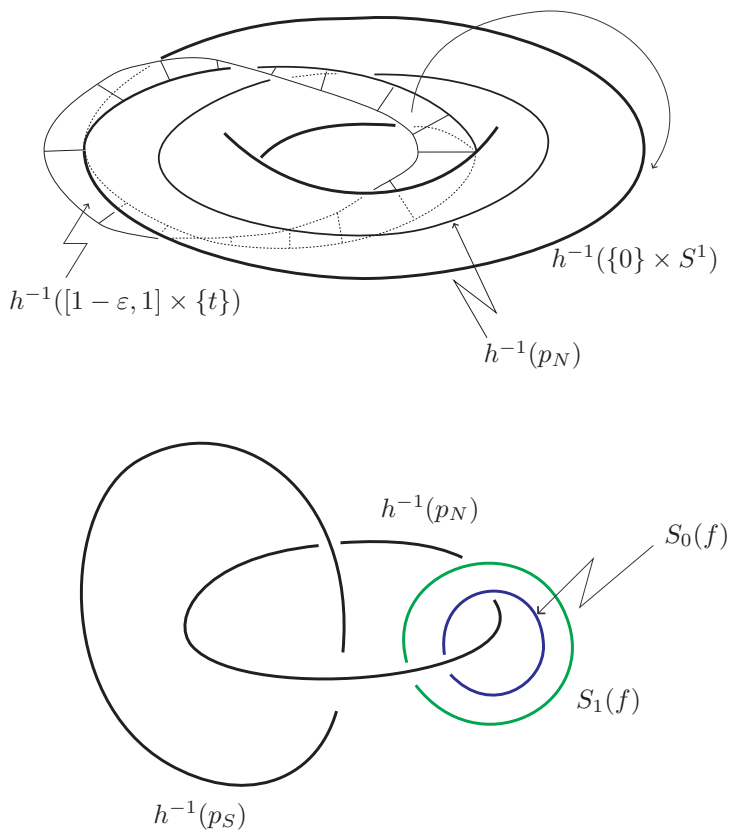}
\caption{Framed regular fiber and the singular point
set of the excellent map $f : S^3 \to S^2$ in
Example~\ref{ex2}}
\label{fig9}
\end{figure}

In this example, the regular fiber component
$h^{-1}(p_S)$ of $f$ does not link with
$S(f)$.
\end{exam}

\begin{exam}
We have yet another example $g : S^3 \to \R^2$
constructed as follows. In the following, we
use the same notations as in Example~\ref{ex2}.
We define $g$ on $h^{-1}(D_N \cup D_S)$ in exactly the same way
as $f$. On the other hand, we replace $f$ on $h^{-1}(A)$
with the map $F$ defined by
$F(x, t) = \eta (k_t(x), t)$ for $x \in S^1 \times [-1, 1]$
and $t \in S^1$, where $\eta : [1, \infty) \times S^1 \to \R^2$
is the embedding as in the above example, and
$k_t : S^1 \times [-1, 1] \to [1, \infty)$, $t \in S^1$,
is a generic $1$--parameter family of functions on
the annulus whose
level sets are as depicted in Fig.~\ref{fig10},
where the red circles depict the boundary components
of the annulus and correspond to the level set $k_t^{-1}(1)$.
Note that for $t \in S^1$, $k_t$ is a Morse function,
except for two values where a birth or a death
of a pair of critical points occurs. 
In the figure, the
blue points depict critical points of index $2$ and the
green ones of index $1$.
The singular value
set of $F$ is as depicted in Fig.~\ref{fig11}, and
the critical points in Fig.~\ref{fig10} correspond to
the curves $\alpha, \beta, \gamma, \delta, \varepsilon$
and $\zeta$ in Fig.~\ref{fig11}.
\newpage

\begin{figure}[t]
\centering
\psfrag{t1}{$t_1$}
\psfrag{t2}{$t_2$}
\psfrag{t3}{$t_3$}
\psfrag{t4}{$t_4$}
\psfrag{a}{$\alpha$}
\psfrag{b}{$\gamma$}
\psfrag{c}{$\beta$}
\psfrag{d}{$\delta$}
\psfrag{e}{$\varepsilon$}
\psfrag{f}{$\zeta$}
\psfrag{birth}{birth of $\beta$ and $\delta$}
\psfrag{death}{death of $\alpha$ and $\gamma$}
\includegraphics[width=\linewidth,height=0.461\textheight,
keepaspectratio]{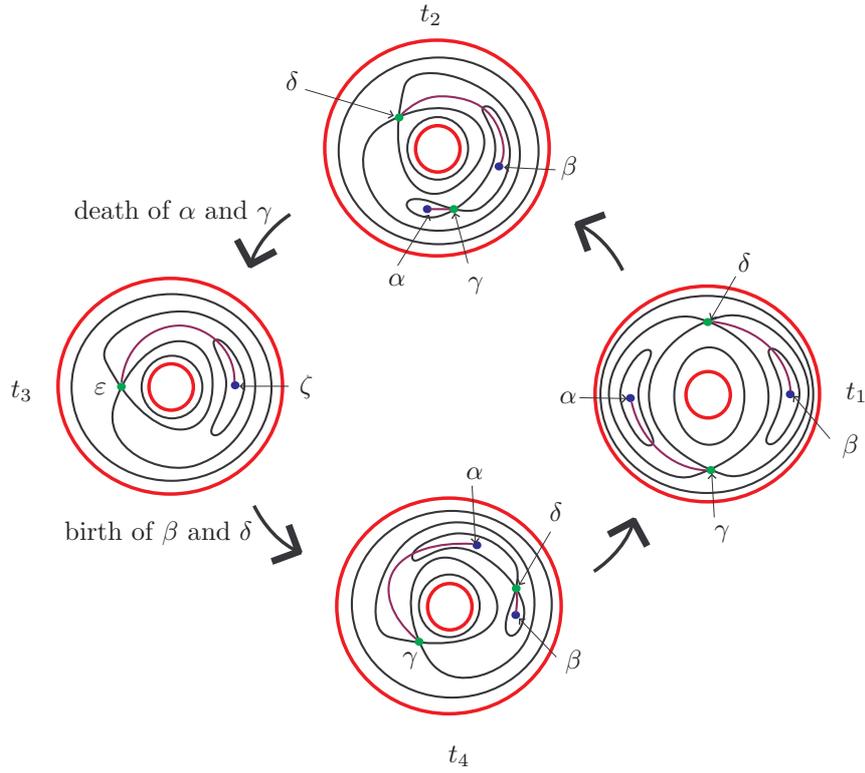}
\caption{Level sets of $k_t : S^1 \times [-1, 1] \to [1, \infty)$ for
$t = t_1, t_2, t_3$ and $t_4 \in S^1$, which correspond
to those in Fig.~\ref{fig11}.}
\label{fig10}
\end{figure}

\begin{figure}[thb]
\centering
\psfrag{t1}{$t_1$}
\psfrag{t2}{$t_2$}
\psfrag{t3}{$t_3$}
\psfrag{t4}{$t_4$}
\psfrag{a}{$\alpha$}
\psfrag{b}{$\gamma$}
\psfrag{c}{$\beta$}
\psfrag{d}{$\delta$}
\psfrag{e}{$\varepsilon$}
\psfrag{f}{$\zeta$}
\includegraphics[width=\linewidth,height=0.3\textheight,
keepaspectratio]{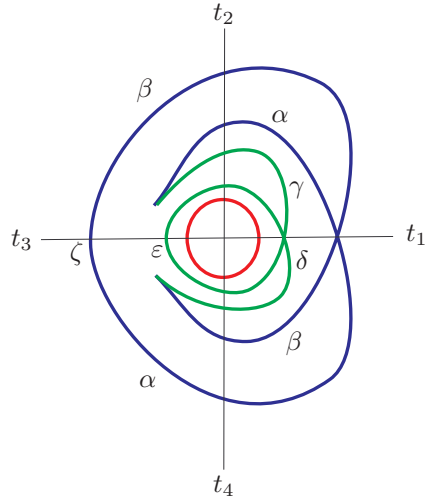}
\caption{Singular value set of $F$, where the red circle
in the center corresponds to the image of
$\eta(\{1\} \times S^1)$, the blue curve corresponds to the
image of the definite fold singularities, and the green one
to the image of the indefinite fold singularities. The values
$t_1,t_2,t_3$ and $t_4 \in S^1$ correspond to those
in Fig.~\ref{fig10}.}
\label{fig11}
\end{figure}
\newpage

In this way, we get an excellent map $g : S^3 \to \R^2$
with exactly two cusp singularities such that $S(g)$
consists of a circle.
Furthermore, we see that $S(g)$ bounds a $2$--disk
in $S^3$ disjoint from the regular fiber
$g^{-1}(\mathbf{0})$. Such a disk can
be found by tracing the purple curves in Fig.~\ref{fig10}.
Therefore, $S(g)$ is an unknotted circle in $S^3$
and is split from the regular fiber over the origin $\mathbf{0}$.
This again gives a desired counter example.
\end{exam}

\begin{rmk}\label{QC}
The above examples show that the answer to the following
question is, in general, negative for excellent maps of $S^3$
into $\R^2$:
\emph{must every component of a regular fiber
be linked by at least one component of the singular 
point set ?} This question was originally posed for
maps of $\R^3$ into $\R^2$, by Professor David Chillingworth
(see \S\ref{section4}).
\end{rmk}

\begin{rmk}
Let $f : M \to \R^2$ be an excellent map
of a closed oriented $3$--manifold $M$.
We assume that $f$ is $C^\infty$ stable, i.e.\ $f|_{S(f)}$
satisfies certain transversality conditions
(for details, see \cite{GG, Levine1}).
Such a $C^\infty$ stable map $f$
is \emph{simple} if it has no cusp singularities
and for every $y \in f(S(f))$, each component of 
$f^{-1}(y)$ contains at most one singular point.
In this case, by \cite{Saeki96}, regular fibers,
the singular point set, or their unions are all
graph links: i.e.\ their exteriors are unions
of circle bundles over surfaces attached along
their torus boundaries. The realization problem
of graph links as regular fibers or the singular point set
has been addressed in \cite{Saeki96}.
See also \cite{Saeki93}.
\end{rmk}

\section{Maps of $\R^3$ into $\R^2$}\label{section4}

The following question was originally posed by
Professor David Chillingworth (see \S\ref{section1} 
and Remark~\ref{QC}).

\begin{prob}\label{prob:C}
For a generic map $f : \R^3 \to \R^2$,
must every component of a regular fiber
be linked by at least one component of 
the singular point set $S(f)$ ?
\end{prob}

In order to answer negatively to the above problem, we use
the following theorem which is due to 
Hector and Peralta-Salas \cite{HP}.

\begin{thm}[Hector and Peralta-Salas, 2012]\label{thm:HP}
Let $L = L_1 \cup L_2 \cup \cdots
\cup L_\mu \subset \R^3$ be an oriented link.
Then, there exist a submersion
$f : \R^3 \to \R^2$
and a regular value $y \in \R^2$ such that
$f^{-1}(y) = L$ if and only if for all $i$
with $1 \leq i \leq \mu$, we have
$$\sum_{j \neq i}\mathrm{lk}(L_i, L_j)
\equiv 1 \pmod 2,$$
where $\mathrm{lk}$ denotes the linking number.
\end{thm}

Now, let
$L$ be a link that satisfies the condition
as described in Theorem~\ref{thm:HP}
(for example, a Hopf link). 
Then, there exist a submersion $f : \R^3 \to \R^2$
and a regular value $y \in \R^2$ with
$L = f^{-1}(y)$.

Take a point $p \in \R^3 \setminus L$
and its small $3$--disk neighborhood $N(p) \subset
\R^3 \setminus L$.
Then, we can deform $f$ in $N(p)$
so that the resulting map $g : \R^3
\to \R^2$ is generic and $S(g)$ is an
unknotted circle in $N(p)$ (use the move called ``lip'' or
``birth''. See \cite[Lemma~3.1]{Saeki95}).
Then, no component of $L = g^{-1}(y)$ 
links with $S(g)$.

This gives a negative answer
to Problem~\ref{prob:C}.

\section*{Acknowledgment}\label{ack}
The author is most thankful to Professor David Chillingworth
for posing the question as in Problem~\ref{prob:C},
which motivated the study developed in this paper. He would also
like to thank Kazuto Takao for helpful discussions.
The author has been supported in part by JSPS KAKENHI Grant Numbers 
JP15K13438, JP16K13754, JP16H03936, JP17H06128.

\end{document}